\documentclass[11pt]{article}

\usepackage{amsfonts}
\usepackage{amssymb,amsmath,amsthm}
\usepackage{latexsym}
\usepackage{fullpage}
\usepackage{hyperref}
\usepackage{graphicx}
\newtheorem{theorem}{Theorem}[section]

\newtheorem{prop}[theorem]{Proposition}

\newtheorem{lemma}[theorem]{Lemma}

\theoremstyle{definition}
\newtheorem{remark}{Remark}

\newcounter{tenumerate}

\def\P{\mathbb{P}}

\newcommand{\deq}{\stackrel{\scriptscriptstyle\triangle}{=}}

\renewcommand{\epsilon}{\varepsilon}

\DeclareMathOperator{\var}{Var}

\newcommand{\E}{{\mathbb E}}
\newcommand{\remove}[1]{}

\renewcommand{\leq}{\leqslant}
\renewcommand{\geq}{\geqslant}

\def\XXint#1#2#3{{\setbox0=\hbox{$#1{#2#3}{\int}$}
\vcenter{\hbox{$#2#3$}}\kern-.5\wd0}}

\def\plus{\mathsf{+}}
\def\minus{\mathsf{-}}
\def\zero{\mathsf{0}}

\title{Exponential decay of correlations in the two-dimensional random field Ising model at zero temperature}
\author{Jian Ding\thanks{Partially supported by NSF grant DMS-1757479 and an Alfred Sloan fellowship.}  \\ University of Pennsylvania \and Jiaming Xia\footnotemark[1]  \\ University of Pennsylvania}

\begin{document}

\maketitle 

\begin{abstract}
We study random field Ising model on $\mathbb Z^2$ where the external field is given by i.i.d.\ Gaussian variables with mean zero and positive variance.  We show that  at zero temperature the effect of boundary conditions on the magnetization in a finite box decays exponentially in the distance to the boundary.
\end{abstract}

\section{Introduction}

For $v\in \mathbb Z^2$, let $h_v$ be i.i.d.\ Gaussian variables with mean zero and variance $\epsilon^2>0$.
We consider random field Ising model (RFIM) with external field $\{h_v: v\in \mathbb Z^2\}$ at zero temperature. 
For $N \geq 1$, let $\Lambda_N = \{v\in \mathbb Z^2: |v|_\infty \leq N\}$ be a box in $\mathbb Z^2$ centered at the origin $o$ and of side length $2N$. For any set $A\subset \mathbb Z^2$, define $\partial A = \{v\in \mathbb Z^2\setminus A: u\sim v\mbox{ for some } u\in  A\}$. The RFIM Hamiltonian $H^{\pm}$ on the configuration space $\{-1, 1\}^{\Lambda_N}$ with plus (respectively, minus) boundary condition and external field $\{h_v: v\in \Lambda_N\}$ is defined to be
\begin{equation}\label{eq-def-mu}
H^\pm(\sigma) = - \big(\sum_{u\sim v, u, v\in \Lambda_N} \sigma_u \sigma_v  \pm \sum_{u\sim v, u\in  \Lambda_N, v\in \partial \Lambda_N} \sigma_u + \sum_{u\in \Lambda_N} \sigma_u h_u\big) \mbox{ for } \sigma \in \{-1, 1\}^{\Lambda_N}\,.
\end{equation}
With probability 1 there exists a unique minimizer, known as the \emph{ground state}, which we denote by $\sigma^{\Lambda_N, +}$ (with respect to the plus-boundary condition) and $\sigma^{\Lambda_N,-}$ (with respect to the minus-boundary condition).
Our main result is the following theorem.
\begin{theorem}\label{thm-main}
For any $\epsilon>0$, there exists $c_\epsilon>0$ such that  $\P(\sigma^{\Lambda_N, +}_o \neq \sigma^{\Lambda_N, -}_o) \leq c_\epsilon^{-1} e^{-c_\epsilon N}$ for all $N\geq 1$.
\end{theorem}

This result lies under the umbrella of the general Imry--Ma \cite{IM75} phenomenon, which states that in two-dimensional systems any first order transition is rounded off upon the introduction of arbitrarily weak static, or quenched,
disorder in the parameter conjugate to the corresponding extensive quantity. In the particular case of RFIM, it was shown in \cite{AW89, AW90} that \emph{for all non-negative temperatures} the effect on the local quenched magnetization of the boundary conditions
at distance $N$ away decays to 0 as $N\to \infty$, which also implies the uniqueness of the Gibbs state. The decay rate was then improved to $1/\log\log N$ in \cite{Chatterjee18} and to $1/N^\gamma$ (for some $\gamma>0$) in \cite{AP18}. All these results apply for arbitrarily weak disorder.  In the presence of high disorder it has been shown that there is an exponential decay \cite{Ber85, FI84, CJN18} (see also \cite[Appendix A]{AP18}). The main remaining challenge is to decide whether the decay rate is exponential when the disorder is weak. In fact,  there have been debates even among physicists on whether there exist regimes where the decay rate is polynomial, and weak supporting arguments have been made in both directions \cite{GMS82, BK87, DS84} --- in particular in \cite{DS84} an argument was made for polynomial decay \emph{at zero temperature} for a certain choice of disorder.
 Our contribution is to prove  exponential decay, \emph{for any $\epsilon>0$ at zero temperature}. The natural remaining question is to try to prove an analogue of Theorem~\ref{thm-main} at positive temperatures.
 
The two-dimensional behavior of RFIM is drastically different from that for dimensions three and higher: it was shown in \cite{Imbrie85} that at zero temperature
the effect on the local quenched magnetization of the boundary conditions at distance $N$ does not vanish in $N$ in the presence of weak disorder, and later an analogous result was proved  in \cite{BK88} at low temperatures. 

Our proof method is different from all of \cite{AW90, Chatterjee18, AP18} (and different from \cite{Ber85, FI84, CJN18}), except that in the heuristic level our proof seems to be related to the Mandelbrot percolation analogy presented in \cite[Appendix B]{AP18}. The works \cite{AW89, AW90} treated a wide class of distributions for disorder, while \cite{Chatterjee18, AP18} and the current paper work with Gaussian disorder. The main features of Gaussian distributions used in the present article are the simple formula for the change of measure (see \eqref{eq-change-of-measure}) and linear decompositions for Gaussian process (see \eqref{eq-Gaussian-conditioning}).

\section{Outline of the proof}

We first reformulate Theorem~\ref{thm-main}.
For $v\in  \Lambda_N$, we define 
\begin{equation}\label{eq-def-xi}
\xi_v^{\Lambda_N} = 
\begin{cases}
\plus, & \mbox{ if } \sigma^{\Lambda_N, +}_v = \sigma^{\Lambda_N, -}_v = 1\,,\\
\minus, & \mbox{ if } \sigma^{\Lambda_N, +}_v = \sigma^{\Lambda_N, -}_v = -1\,,\\
\zero, & \mbox{ if }   \sigma^{\Lambda_N, +}_v = 1 \mbox{ and } \sigma^{\Lambda_N, -}_v = -1\,.
\end{cases}
\end{equation}
By monotonicity (c.f. \cite[Section 2.2]{AP18}), the case of  $\sigma^{\Lambda_N, +}_v = -1 \mbox{ and } \sigma^{\Lambda_N, -}_v = 1$ cannot occur, so $\xi_v^{\Lambda_N}$ is well-defined for all $v\in \Lambda_N$.  Theorem~\ref{thm-main} can be restated as 
\begin{equation}\label{eq-main-result}
m_N \leq c_\epsilon^{-1} e^{-c_\epsilon N} \mbox{ for } c_\epsilon>0, \mbox{ where } m_N \deq \P(\xi_{o}^{\Lambda_N}=\zero) \,.
\end{equation}
For any $A\subset \mathbb Z^2$,  we can analogously define $\xi^A$ by replacing $\Lambda_N$ with $A$ in \eqref{eq-def-mu} and \eqref{eq-def-xi}.  Let $\mathcal C^{A} = \{v\in A: \xi_v^{A} = \zero\}$.
Monotonicity (see \cite[(2.7)]{AP18}) implies that
\begin{equation}\label{eq-monotonicity}
\mathcal C^{B} \cap B' \subset \mathcal C^{B'} \mbox{ provided that } B' \subset B\,.
\end{equation}
In particular, this implies that $m_N$ is decreasing in $N$, so we need only consider $N=2^n$ for $n\geq 1$.
Clearly, for any $v\in \mathcal C^A$, there exists a path in $\mathcal C^A$ joining $v$ and $\partial A$. This suggests consideration of percolation properties of $\mathcal C^A$. Indeed, a key step in our proof for \eqref{eq-main-result} is the following proposition on the lower bound on the length exponent for geodesics (i.e., shortest paths) in $\mathcal C^{\Lambda_N}$. For any $A\subset \mathbb Z^2$, we denote by $d_A (\cdot, \cdot)$ the graph distance on the induced subgraph on $A$. 
\begin{prop}\label{prop-crossing-dimension}
There exist $\alpha = \alpha(\epsilon) > 1$, $\kappa = \kappa(\epsilon)>0$ such that for all $N\geq 1$
\begin{equation}\label{eq-box-counting}
\P(d_{\mathcal C^{\Lambda_N}} (\partial \Lambda_{N/4}, \partial \Lambda_{N/2}) \leq N^{\alpha}) \leq \kappa^{-1} e^{- N^{\kappa}}\,.
\end{equation}
\end{prop}
The proof of Proposition~\ref{prop-crossing-dimension} will rely on \cite{AB99}, which takes  the next lemma as input. For any rectangle $A\subset \mathbb R^2$ (whose sides are not necessarily parallel to the axes), let $\ell_A$ be the length of the longer side and let $ A^{\mathrm{Large}}$ be the square box concentric with $A$ and of side length $32\ell_A$. In addition, define the aspect ratio of $A$ to be the ratio between the lengths of the longer and shorter sides. For a (random) set $\mathcal C \subset \mathbb Z^2$, we use $\mathrm{Cross}(A, \mathcal C)$ to denote the event that there exists a path $v_0, \ldots, v_k \in A \cap \mathcal C$ connecting the two shorter sides of $A$ (that is,  $v_0, v_k$ are of $\ell_\infty$-distances less than 1 respectively from the two shorter sides of $A$).
\begin{lemma}\label{lem-assumption}
Write $a = 100$. There exists $\ell_0 = \ell_0(\epsilon)$ and $\delta = \delta(\epsilon) >0$ such that the following holds for any $N\geq 1$. For any $k\geq 1$ and any rectangles $A_1, \ldots, A_k \subseteq \Lambda_{N/2}$ with aspect ratios at least $a$ such that  (a) $\ell_0\leq \ell_{A_i} \leq N/32$ for all $1\leq i\leq k$ and (b) $A^{\mathrm{Large}}_1, \ldots, A^{\mathrm{Large}}_k$ are disjoint, we have
$$\P(\cap_{i=1}^k \mathrm{Cross}(A_i, \mathcal C^{\Lambda_N})) \leq (1-\delta)^k\,.$$
\end{lemma}
(Actually, the authors of \cite{AB99} treated random curves in $\mathbb R^2$. However, the main capacity analysis can be copied in the discrete case, and the connection between the capacity and the box-counting dimension is straightforward (c.f. \cite[Lemma 2.3]{DT16}).) Armed with Lemma~\ref{lem-assumption}, we can apply \cite[Theorem 1.3]{AB99} to deduce that for some $\alpha = \alpha(\epsilon)>1$,
\begin{equation}\label{eq-box-counting-weak}
\P(d_{\mathcal C^{\Lambda_N}} (\partial \Lambda_{N/4}, \partial \Lambda_{N/2}) \leq N^{\alpha}) \to 0 \mbox{ as } N \to \infty\,.
\end{equation}
By a standard percolation argument (Lemma~\ref{lem-enhance}) which we will explain later, we can enhance the probability decay in \eqref{eq-box-counting-weak} and prove \eqref{eq-box-counting}.

By \eqref{eq-monotonicity},
the random set
$\mathcal C^{\Lambda_N} \cap A$ is stochastically dominated by  $\mathcal C^{A^{\mathrm{Large}}} \cap A$ as long as $A^{\mathrm{Large}}\subset \Lambda_N$. Moreover, it is obvious that $\mathcal C^{A^{\mathrm{Large}}_i}$ for $1\leq i\leq k$ are mutually independent, as long as the sets $A_i^{\mathrm{Large}}$ for $1\leq i\leq k$ are disjoint. Therefore, in order to prove Lemma~\ref{lem-assumption}, it suffices to show that for any rectangle $A$ with aspect ratio at least $a = 100$ we have

\begin{equation}\label{eq-crossing-prob}
\P(\mathrm{Cross}(A, \mathcal C^{A^{\mathrm{Large}}})) \leq 1- \delta \mbox{ where } \delta=  \delta(\epsilon) >0\,.
\end{equation}

Both the proof of \eqref{eq-crossing-prob} and the application of \eqref{eq-box-counting} rely on a perturbative analysis, which is another key feature of our proof. Roughly speaking, the logic is as follows:
\begin{itemize}
\item We first consider the perturbation by increasing the field by an amount of order $1/N$, and use this to show that the probability for a $\zero$-valued contour surrounding an annulus is strictly bounded away from 1.
\item Based on this property, we prove  \eqref{eq-crossing-prob}, which then implies \eqref{eq-box-counting}.
\item Given \eqref{eq-box-counting}, i.e., that the length exponent for the geodesic is at least $\alpha>1$, we then show that increasing the field by an amount of order $1/N^\alpha$ will most likely change the $\zero$'s to $\plus$'s. Based on this, we prove polynomial decay for $m_N$ with large power, which can then be enhanced to exponential decay.
\end{itemize}
For compactness of exposition, the actual implementation will differ slightly from the above plan:
\begin{itemize}
\item We first prove a general perturbation result in Section~\ref{sec-perturbation}, where the size of perturbation is related to the graph distance on the induced graph on $\mathcal C^{\Lambda_N}$.
\item In Section~\ref{sec-crossing-dimension}, we apply Lemma~\ref{lem-perturbation} by  bounding  $d_{\mathcal C^{\Lambda_N}}$ from below by the $\ell_1$-distance and correspondingly setting the perturbation amount to $1/N$, thereby proving  Lemma~\ref{lem-bound-hard-crossing}. As a consequence, we verify \eqref{eq-crossing-prob}.
\item In Section~\ref{sec-main-thm}, we apply Lemma~\ref{lem-perturbation} again by applying a lower bound on  $d_{\mathcal C^{\Lambda_N}}$ from Proposition~\ref{prop-crossing-dimension}. This allows us to derive Lemma~\ref{lem-m-star}. As a consequence, we prove in Lemma~\ref{lem-m-N-bound}  polynomial decay for $m_N$ with large power, which is then  enhanced to exponential decay by a standard argument. 
\end{itemize}

\section{A perturbative analysis}\label{sec-perturbation}

We first introduce some notation. For $A\subseteq \mathbb Z^2$, we set $h_A = \sum_{v\in A} h_v$. For $A, B \subset \mathbb Z^2$, we denote by $E(A, B) = \{\langle u, v\rangle: u\sim v, u\in A, v\in B\}$. Note that we treat $\langle u, v\rangle$ as an ordered edge. For simplicity, we will only consider $N = 2^n$ for $n\geq 10$. Let $\mathcal A_N = \Lambda_N \setminus \Lambda_{N/2}$ be an annulus.   In what follows, we will denote $\{\tilde h^{(N)}_v: v\in \Lambda_N\}$ as various perturbations of the original field, whose meaning will depend on the context. In all situations  we will use  $\tilde H^\pm(\sigma)$, $\tilde \sigma^{\Lambda_N, \pm}$, $\tilde \xi^{\Lambda_N}$, $\tilde {\mathcal C}^{\Lambda_N}$ to denote the corresponding tilde versions of $H^\pm(\sigma)$, $\sigma^{\Lambda_N, \pm}$, $\xi^{\Lambda_N}$, ${\mathcal C}^{\Lambda_N}$, i.e., defined analogously but with respect to the field $\{\tilde h^{(N)}_v\}$. In addition, define $\mathcal C_*^{\Lambda_N} = \tilde {\mathcal C}^{\Lambda_N} \cap  {\mathcal C}^{\Lambda_N}$. 
\begin{lemma}\label{lem-perturbation}
Consider $K, \Delta > 0$. Define 
\begin{equation}\label{eq-def-tilde-h}
\tilde h^{(N)}_v  = 
h_v + \Delta \mbox{ for } v\in \Lambda_N\,.
\end{equation}

The following two conditions cannot hold simultaneously:

(a) $d_{\mathcal C_*^{\Lambda_N}}(\partial \Lambda_{N/4}, \partial \Lambda_{N/2}) \geq K$;

(b) $|\mathcal C_*^{\Lambda_N} \cap \Lambda_{N/4}| \cdot \Delta> \frac{8}{K} |\mathcal C_*^{\Lambda_N} \cap \mathcal A_{N/2}|$.
\end{lemma}
\begin{proof}
Suppose otherwise both (a) and (b) hold. Let  $B_k = \{v\in \mathcal A_{N/2}: d_{\mathcal C_*^{\Lambda_N}}(\partial \Lambda_{N/4}, v) = k\}$, for $k=1, \ldots, K$. Note that $B_k \subset \mathcal C_*^{\Lambda_N} \cap \mathcal A_{N/2}$ for all $1\leq k\leq K$ by (a). It is obvious that the $B_k$'s are disjoint from each other, and thus there exists a minimal value $k_*$ such that 
\begin{equation}\label{eq-B-*}
|B_{k_*}| \leq K^{-1}| \mathcal C_*^{\Lambda_N} \cap \mathcal A_{N/2}|\,.
\end{equation}
 Let 
$$S =(\mathcal C_*^{\Lambda_N} \cap \Lambda_{N/4} ) \cup \cup_{k=1}^{k^*-1} B_k\,,$$
and for $\tau \in \{\minus, \zero, \plus\}$, define
\begin{equation}\label{eq-definition-g}
g(S, \tau) = \{\langle u, v\rangle \in E(S, S^c): \xi^{\Lambda_N}_v = \tau \} \mbox{ and } \tilde g(S, \tau) = \{\langle u, v\rangle \in E(S, S^c): \tilde \xi^{\Lambda_N}_v = \tau \}\,.
\end{equation}
Note that for any $v\in \Lambda_N$ with $\xi_v^{\Lambda_N} = \zero$ we have $\sigma^{+, \Lambda_N}_v = 1$. Since $\xi^{\Lambda_N}_{v} = \zero$ for $v\in S$ (which implies that $\sigma^{\plus, \Lambda_N}_v = 1$ for $v\in S$), 
\begin{equation}\label{eq-S-condition}
h_S + |g(S, \plus)| - |g(S, \minus)| + |g(S, \zero)| \geq 0\,,
\end{equation}
because if \eqref{eq-S-condition} does not hold, then $H^\plus(\sigma') < H^\plus(\sigma^{\plus, \Lambda_N})$ where $\sigma'$ is obtained from $\sigma^{\plus, \Lambda_N}$ by flipping its value on $S$, thus contradicting the minimality of $H^+(\sigma^{\plus, \Lambda_N})$. In addition, by monotonicity (with respect to the external field), we have $g(S, \zero) \subset \tilde g(S, \zero)\cup \tilde g(S, \plus)$, $g(S, \plus) \subset \tilde g (S, \plus)$, and thus
$$|\tilde g(S, \plus)| - |g(S, \plus)| \geq | g(S, \zero) \setminus \tilde g(S, \zero)|\,.$$
Similarly we have $\tilde g(S, \minus) \subset g(S, \minus)$ and $\tilde g(S, \zero) \subset g(S, \minus) \cup g(S, \zero)$, and thus
$$|g(S, \minus)| - |\tilde g(S, \minus)| \geq |\tilde g(S, \zero) \setminus g(S, \zero)|\,.$$
 By our definition of $B_k$'s, we see that $\tilde g(S, \zero) \cap g(S, \zero) = E(S, B_{k_*})$. Therefore, \eqref{eq-S-condition} and the preceding two displays imply that
\begin{align*}
\tilde h^{(N)}_S + |\tilde g(S, \plus)| - |\tilde g(S, \minus)| - |\tilde g(S, \zero)| &\geq \tilde h^{(N)}_S + | g(S, \plus)| - | g(S, \minus)| + |g(S, \zero)| - 2| E(S, B_{k_*})|\\
&\geq |S| \Delta - 8|B_{k_*}| >0\,,
\end{align*}
where the last inequality follows from (b) and \eqref{eq-B-*}. The preceding inequality implies that $\tilde H^-(\sigma') < \tilde H^\minus(\tilde \sigma^{\minus, \Lambda_N})$ where $\sigma'$ is obtained from $\tilde \sigma^{\minus, \Lambda_N}$ by flipping its value on $S$. This contradicts the minimality of $\tilde H^\minus(\tilde \sigma^{\minus, \Lambda_N})$, completing the proof of the lemma.
\end{proof}

\begin{lemma}\label{lem-percolation-property}
For any $x_v \geq 0$ for  $v\in \Lambda_N$, let $\tilde h^{(N)}_v = h_v + x_v$ for $v\in \Lambda_N$. Then with probability 1, for any $v\in \mathcal C_*^{\Lambda_N}$ there is a path in $ \mathcal C_*^{\Lambda_N}$ joining $v$ and $\partial \Lambda_N$.
\end{lemma}
\begin{proof}
The proof is similar to that of Lemma~\ref{lem-perturbation}, and in a way it is the case of $K = \infty$ there.

Suppose that the claim is not true. Then take $v\in \mathcal C_*^{\Lambda_N}$ (for which the claim fails), and let $S$ be the connected component in $\mathcal C_*^{\Lambda_N}$ that contains $v$ (thus $S$ is not neighboring $\partial \Lambda_N$). Define $g(S, \tau)$ and $\tilde g(S, \tau)$ as in \eqref{eq-definition-g}. Similar to \eqref{eq-S-condition}, we have that
$$h_S + |g(S, \plus)| - |g(S, \minus)| + |g(S, \zero)| \geq 0\,.$$
In our case, $g(S, \zero) \cup g(S, \plus) \subset  \tilde g(S, \plus)$ and $\tilde g(S, \zero) \cup \tilde g(S, \minus) \subset g(S, \minus)$. Therefore,
$$\tilde h^{(N)}_S + |\tilde g(S, \plus)| - |\tilde g(S, \minus)| - |\tilde g(S, \zero)| \geq  h_S + | g(S, \plus)| - | g(S, \minus)| + |g(S, \zero)| \geq 0\,.$$
The preceding inequality implies that $\tilde H^-(\sigma') \leq \tilde H^\minus(\tilde \sigma^{\minus, \Lambda_N})$ where $\sigma'$ is obtained from $\tilde \sigma^{\minus, \Lambda_N}$ by flipping its value on $S$. This happens with probability 0 since the ground state is unique with probability 1.
\end{proof}

\section{Proof of Proposition~\ref{prop-crossing-dimension}}\label{sec-crossing-dimension}

In this section, we will set $K = K(N)= N/4$, and $\Delta = \Delta(N)= \gamma/N$ for an absolute constant $\gamma>0$ to be selected, and we consider $\tilde h^{(N)}$ as in \eqref{eq-def-tilde-h}. 
In this case Condition (a) in Lemma~\ref{lem-perturbation} holds trivially. For convenience, we use $ \P_N$ to denote the probability measure with respect to the field $\{ h_v: v\in \Lambda_N\}$ and use $\tilde \P_N$ to denote the probability measure with respect to $\{\tilde h^{(N)}_v: v\in \Lambda_N\}$.
\begin{lemma}\label{lem-perturbation-contiguous}
Recall that $\epsilon$ is the variance parameter for the field $\{h_v\}$. For any $p>0$, there exists $c = c(\epsilon,p, \gamma)>0$ such that for any event $E_N$ with $\tilde \P_N(E_N)\geq p$, we have that
$$\P_N(E_N) \geq c\,.$$
\end{lemma}
\begin{proof}
There exists a constant $C>0$ such that $\tilde \P_N(|\tilde h^{(N)}_{\Lambda_{N}} - \Delta |\Lambda_N|| \geq C \epsilon N) \leq p/2$. Thus we have 
\begin{equation}\label{eq-prob-E-N-bound}
\tilde \P_N(E_N; |\tilde h^{(N)}_{\Lambda_{N}} - \Delta |\Lambda_N|| \leq C \epsilon N) \geq p/2\,.
\end{equation}
Also, by a straightforward Gaussian computation, we see that
\begin{equation}\label{eq-change-of-measure}
\frac{d  \P_N}{d \tilde \P_N}  = \exp\big\{ - \frac{\Delta(\tilde h^{(N)}_{\Lambda_{N}} - \Delta |\Lambda_N|)}{\epsilon^2} \big\} \exp\big\{\frac{-\Delta^2 |\Lambda_N|}{2\epsilon^2}\big\}
\end{equation}
and thus there exists $\iota = \iota (\epsilon)>0$ such that 
$$\frac{d  \P_N}{d \tilde \P_N} \geq \iota \mbox{ provided that } |\tilde h^{(N)}_{\Lambda_{N}} - \Delta |\Lambda_N|| \leq C \epsilon N\,.$$
Combined with \eqref{eq-prob-E-N-bound}, this completes the proof of the lemma.
\end{proof}
For any annulus $\mathcal A$, we denote by $\mathrm{Cross}_{\mathrm{hard}}(\mathcal A, \mathcal C)$ the event that there is a contour in $\mathcal C$ which separates the inner and outer boundaries of $\mathcal A$, and by $\mathrm{Cross}_{\mathrm{easy}}(\mathcal A, \mathcal C)$ the event that there is a path in $\mathcal C$ which connects the inner and outer boundaries of $\mathcal A$.
\begin{lemma}\label{lem-bound-hard-crossing}
There exists $\delta = \delta(\epsilon)>0$ such that  
$$
\min\{\P(\mathrm{Cross}_{\mathrm{hard}}(\Lambda_{N/8}\setminus \Lambda_{N/32}, \mathcal C^{\Lambda_N})), \P(\mathrm{Cross}_{\mathrm{easy}}(\Lambda_{N/8}\setminus \Lambda_{N/32}, \mathcal C^{\Lambda_N})) \}\leq 1- \delta \mbox{ for all } N \geq 32.$$
\end{lemma}
\begin{proof}
We can write $\mathcal A_{N/2} = \cup_{i=1}^r A_i$ where each $A_i$ is a  box of side length $N/16$ (so a copy of $\Lambda_{N/32}$) and $r\geq 16$ is a fixed integer. For a box $A$, denoting by $A^{\mathrm{Big}}$ as the concentric box of $A$ whose side length is $4 \ell_A$. We have that 
\begin{equation}\label{eq-disjointness}
A_i^{\mathrm{Big}} \cap \Lambda_{N/8} = \emptyset  \mbox{ and } A_i^{\mathrm{Big}}  \subset \Lambda_N \mbox{ for all } 1\leq i\leq r.
\end{equation}
For any $A\subset \Lambda_N$, let $\bar {\mathcal C}^A$ be defined as $\mathcal C^A$ but replacing $\{h_v: v\in A\}$ by $\{\tilde h^{(N)}_v: v\in A\}$ (note that $\bar {\mathcal C}^{\Lambda_{N/2}}$ is different from $\tilde {\mathcal C}^{\Lambda_{N/2}}$, which is defined with respect to $\tilde h^{(N/2)}$). Write $\mathcal C_\diamond^{A} = \mathcal C^A \cap \bar {\mathcal C}^A$.
Write $X_i = |\mathcal C_\diamond^{A_i^{\mathrm{Big}}} \cap A_i|$ and $X = |\mathcal C_\diamond^{\Lambda_{N/8}} \cap \Lambda_{N/32}|$. Clearly $X_i$'s and $X$ are identically distributed and by \eqref{eq-disjointness} $X_i$'s are independent of $X$ (but $X_i$'s are not mutually independent). Let $\theta = \inf\{x: \P(X \leq x) \geq 1 - 1/2r \}$. Thus, 
\begin{equation}\label{eq-prob-E}
\mathbb P(X \geq \max_{1\leq i\leq r} X_i, X\geq \theta) \geq \P(X\geq \theta) \P(\max_{1\leq i\leq r} X_i \leq \theta) \geq 1/4r\,.
\end{equation}
The rest of the proof divides into two cases.

\noindent {\bf Case 1:} $\theta>0$.
Let $\mathcal E = \{  |\mathcal C_\diamond^{\Lambda_{N/8}} \cap \Lambda_{N/32}| \geq r^{-1} | \mathcal C_*^{\Lambda_N} \cap \mathcal A_{N/2}|\} \cap\{ |\mathcal C_\diamond^{\Lambda_{N/8}} \cap \Lambda_{N/32}|  >0\}$. By \eqref{eq-monotonicity} and \eqref{eq-disjointness}, we have $|\mathcal C_*^{\Lambda_N} \cap \mathcal A_{N/2}| \leq \sum_{i=1}^r X_i$. Combined with \eqref{eq-prob-E}, it gives that $\P(\mathcal E) \geq 1/4r$. Setting $\gamma = 100 r$, we get that $|\mathcal C_\diamond^{\Lambda_{N/8}} \cap \Lambda_{N/32}| \cdot \Delta> 16 K^{-1}  | \mathcal C_*^{\Lambda_N} \cap \mathcal A_{N/2}|$ on $\mathcal E$. By Lemma~\ref{lem-perturbation}, on $\mathcal E$ there is at least one vertex $v\in \mathcal C_\diamond^{\Lambda_{N/8}} \cap \Lambda_{N/32}$ but $v\not\in \mathcal C_*^{\Lambda_N}$. So either $v \not\in \mathcal C^{\Lambda_N} $ or $v \not\in \tilde {\mathcal C}^{\Lambda_N} $ on $\mathcal E$. Assume that $v \not\in \mathcal C^{\Lambda_N}$ and the other case can be treated similarly. 

We will use the following property: for any connected set $\mathcal A$, $u\not\in \mathcal C^{\mathcal A}$ if and only if there exists a connected set $A\subset \mathcal A$ with $u\in A$ such that $\xi^A_w= \plus$ for all $w\in A$ or $\xi^A_w= \minus$ for all $w\in A$. The ``if'' direction of the property follows from \eqref{eq-monotonicity}. For the ``only if'' direction, we assume without loss that $\xi_u^{\mathcal A} = \plus$ and let $A$ be the connected component containing $u$ where the $\xi^{\mathcal A}$-value is $\plus$. Note $\sigma^{\mathcal A, -}_w = -1$ for all $w\in \partial A$ and $\sigma^{\mathcal A, -}_w = 1$ for all $w\in A$. This implies that $\xi^A_w = \plus$ for all $w\in A$.

By the preceding property, there exists a connected set $A\subset \Lambda_N$ with $v\in A$ such that $\xi^{A}_w = \plus$ for all $w\in A$ or $\xi^{A}_w = \minus$ for all $w\in A$. In addition, $A$ cannot be contained in $\Lambda_{N/8}$ since otherwise it contradicts $v\in \mathcal C^{\Lambda_{N/8}}$. By planar duality, this implies that on $\mathcal E$, either $\mathrm{Cross}_{\mathrm{hard}}(\Lambda_{N/8}\setminus \Lambda_{N/32}, \mathcal C^{\Lambda_N})$ or $\mathrm{Cross}_{\mathrm{hard}}(\Lambda_{N/8}\setminus \Lambda_{N/32}, \tilde {\mathcal C}^{\Lambda_N})$ does not occur (the second case corresponds to the case when $v\not\in \tilde {\mathcal C}^{\Lambda_N}$). Therefore, 
$$\P((\mathrm{Cross}_{\mathrm{hard}}(\Lambda_{N/8}\setminus \Lambda_{N/32}, \mathcal C^{\Lambda_N}))^c) + \P((\mathrm{Cross}_{\mathrm{hard}}(\Lambda_{N/8}\setminus \Lambda_{N/32}, \tilde {\mathcal C}^{\Lambda_N}))^c) \geq \P(\mathcal E) \geq 1/4r\,.$$
Combined with Lemma~\ref{lem-perturbation-contiguous}, this completes the proof of the lemma.

\noindent {\bf Case 2:}  $\theta = 0$. Applying a simple union bound (by using 16 copies of $\Lambda_{N/32}$ to cover $\Lambda_{N/8}$, and a similar derivation to $|\mathcal C_*^{\Lambda_N} \cap \mathcal A_{N/2}| \leq \sum_{i=1}^r X_i$) we get that $\P(\mathcal C_*^{\Lambda_N} \cap  \Lambda_{N/8} = \emptyset) \geq 1/2$. We assume without loss that $ \P(\mathrm{Cross}_{\mathrm{easy}}(\Lambda_{N/8}\setminus \Lambda_{N/32}, \mathcal C^{\Lambda_N})) \geq 3/4$ (otherwise there is nothing further to prove), and thus 
$$ \P(\mathrm{Cross}_{\mathrm{easy}}(\Lambda_{N/8}\setminus \Lambda_{N/32}, \mathcal C^{\Lambda_N})\mbox{ and }\mathcal C_*^{\Lambda_N} \cap  \Lambda_{N/8} = \emptyset ) \geq 1/4\,.$$
On the event $\mathrm{Cross}_{\mathrm{easy}}(\Lambda_{N/8}\setminus \Lambda_{N/32}, \mathcal C^{\Lambda_N})\mbox{ and }\mathcal C_*^{\Lambda_N} \cap  \Lambda_{N/8} = \emptyset$, the easy crossing (joining two boundaries of $\Lambda_{N/8}\setminus \Lambda_{N/32}$) in $\mathcal C^{\Lambda_N}$ becomes an easy crossing with $\tilde \xi^{\Lambda_N}$-values $\plus$, and thus by planar duality prevents existence of a contour surrounding $\Lambda_{N/32}$ in $(\Lambda_{N/8}\setminus \Lambda_{N/32}) \cap \tilde {\mathcal C}^{\Lambda_N}$. Therefore, 
$$\P((\mathrm{Cross}_{\mathrm{hard}}(\Lambda_{N/8}\setminus \Lambda_{N/32}, \tilde {\mathcal C}^{\Lambda_N}))^c) \geq 1/4\,.$$
Combined with Lemma~\ref{lem-perturbation-contiguous}, this completes the proof of the lemma.
\end{proof}

\begin{proof}[Proof of \eqref{eq-crossing-prob}]
Let $N = \min\{2^n: 2^{n+2} \geq \ell_A\}$. By our assumption on $A$, it is clear that we can position four copies $A_1, A_2, A_3, A_4$ of $A$ by translation or rotation by 90 degrees so that (see the left of Figure~\ref{figure})
\begin{itemize}
\item $A_1, A_2, A_3, A_4 \subset \Lambda_{N/8} \setminus \Lambda_{N/32}$.
\item  The union of any crossings through $A_1, A_2, A_3, A_4$ in their longer directions surrounds $\Lambda_{N/32}$.
\item $\Lambda_N \subset A^{\mathrm{Large}}_i$ for $1\leq i\leq 4$.
 \end{itemize}
Set $p = \P(\mathrm{Cross}(A, \mathcal C^{A^{\mathrm{Large}}}))$ (note that $p$ depends on the dimension of $A$ and also the orientation of $A$). By rotation symmetry and \eqref{eq-monotonicity} we see that $\P(\mathrm{Cross}(A_i, \mathcal C^{\Lambda_N}) ) \geq \P(\mathrm{Cross}(A_i, \mathcal C^{A_i^{\mathrm{Large}}}) ) =   p$.  In what follows, we denote $\mathcal A =  \Lambda_{N/8} \setminus \Lambda_{N/32}$. Then, by  $\P(\mathrm{Cross}(A_i, \mathcal C^{\Lambda_N}) ) \geq p$ and a simple union bound, we get that
 \begin{equation}\label{eq-cross-A-hard}
 \P(\mathrm{Cross}_{\mathrm{hard}}(\mathcal A, \mathcal C^{\Lambda_N})) \geq \P (\cap_{i=1}^4 \mathrm{Cross}(A_i, \mathcal C^{\Lambda_N}))\geq 1 - 4(1-p)\,.
 \end{equation}
 Similarly, we can arrange two copies  $A_a, A_b$ of $A$ obtained by translation and rotation by 90 degrees such that $\Lambda_N \subset A_a^{\mathrm{Large}}, A_b^{\mathrm{Large}}$ and that the union of any two crossings through $A_a^{\mathrm{Large}}, A_b^{\mathrm{Large}}$ in the longer direction connects the two boundaries of $\mathcal A$ (see the right of Figure~\ref{figure}). This implies that
 \begin{equation}\label{eq-cross-A-easy}
 \P(\mathrm{Cross}_{\mathrm{easy}}(\mathcal A, \mathcal C^{\Lambda_N})) \geq  \P (\mathrm{Cross}(A_a, \mathcal C^{\Lambda_N}) \cap \mathrm{Cross}(A_b, \mathcal C^{\Lambda_N}))\geq 1 - 2(1-p)\,.
 \end{equation}
 Combined with \eqref{eq-cross-A-hard} and Lemma~\ref{lem-bound-hard-crossing}, it yields that $p\leq 1 - \delta$ for some $\delta = \delta(\epsilon)>0$ as required.
\end{proof}
 
The following standard lemma will be applied several  times below. Divide $\Lambda_N$ into disjoint boxes of side lengths $N' \leq N$ where $N' = 2^{n'}$ for some $n'\geq 1$, and denote by $\mathcal B(N, N')$ the collection of such boxes. Consider a percolation process on $\mathcal B(N, N')$, where each box $B\in \mathcal B(N, N')$ is regarded open or closed randomly. For $C, p>0$, we say that the percolation process satisfies the $(N, N', C, p)$-condition if for each $B\in \mathcal B(N, N')$, there exists an event $E_B$ such that 
\begin{itemize}
\item On $E_B^c$, $B$ is closed.
\item $\P(E_B) \leq p$ for each $B$.
\item If $\min_{x\in B_i, y\in B_j}|x-y|_\infty \geq CN'$ for all $1\leq i<j\leq k$, then the events $E_{B_1}, \ldots, E_{B_k}$ are  mutually independent.
\end{itemize}
Furthermore, we say two boxes $B_1, B_2$ are adjacent if $\min_{x_1\in B_1, x_2\in B_2} |x_1 - x_2|_\infty  \leq 1$, and we say a collection of boxes is a lattice animal if these boxes form a connected graph.
\begin{lemma}\label{lem-enhance}
For any $C>0$, there exists $p>0$ such that for all $N$ and $N'\leq N$ and any percolation process on $\mathcal B(N, N', C, p)$ satisfying the $(N,N', C, p)$-condition, we have
$$\P(\mbox{there exists a lattice animal of open boxes on } \mathcal B(N, N') \mbox{ of size  at least } k) \leq (\tfrac{N}{N'})^2 2^{-k}\,.$$
\end{lemma}
\begin{proof}
 On the one hand, the number of lattice animals of size exactly $k$ is bounded by  $(\tfrac{N}{N'})^2 8^{2k}$ (the bound comes from first choosing a starting box, and then encoding the lattice animal by a  surrounding contour on $ \mathcal B(N, N')$ of length $2k$). On the other hand, for any $k$ such boxes, we can extract a sub-collection of $ck$ boxes (here $c>0$ is a constant that depends only on $C$) such that the pairwise distances of boxes in this sub-collection are at least $CN'$; hence the probability that all these $k$ boxes are open is at most $p^{ck}$. The proof of the lemma is then completed by a simple union bound, employing the $(N, N', C, p)$-condition.
\end{proof}

 \begin{figure}[h] 
  \includegraphics[width=15cm]{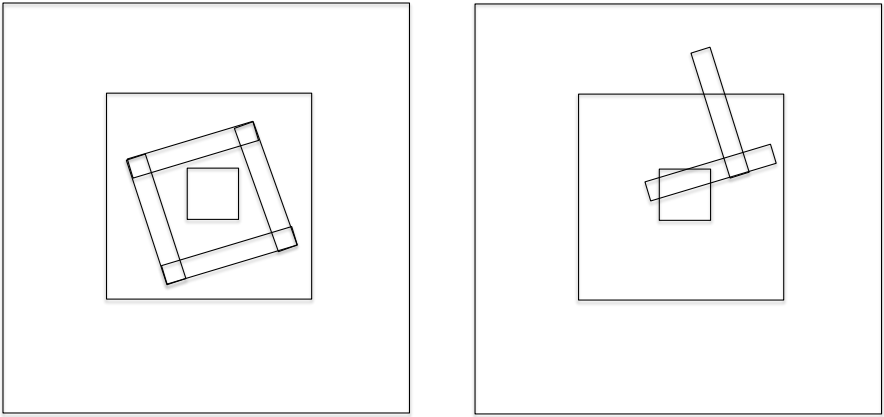}
\\ \vspace{-5.5cm} \caption{On both left and right, the three concentric square boxes are $\Lambda_N$, $\Lambda_{N/8}$ and $\Lambda_{N/32}$ respectively. On the left, the four rectangles are $A_1, A_2, A_3, A_4$ and on the right the two rectangles are $A_a, A_b$.}\label{figure}
\end{figure}

\begin{proof}[Proof of Proposition~\ref{prop-crossing-dimension}]
Let $N'= N^{1 - (\frac{\alpha - 1}{10} \wedge \frac{1}{10})}$, where $\alpha$ is as in \eqref{eq-box-counting-weak}. For each $B\in \mathcal B(N, N')$, we say $B$ is open if $d_{\mathcal C^{B^{\mathrm{Large}}}} (\partial B, \partial B^{\mathrm{large}}) \leq (N')^{\alpha}$, where $B^{\mathrm{large}}$ is the box  concentric with $B$ of doubled side length and $B^{\mathrm{Large}}$ (as we recall) is a concentric box of $B$ with side length $8\ell_B$. By \eqref{eq-box-counting-weak}, we see that  this percolation process satisfies the $(N, N', 16,p)$-condition where $p \to 0$ as $N\to \infty$. Now,  in order that $ d_{\mathcal C^{\Lambda_N}} (\partial \Lambda_{N/4}, \partial \Lambda_{N/2}) \leq (N')^{\alpha}$, there must exist an open lattice animal on  $\mathcal B(N, N')$ of size at least $\frac{N}{16N'}$. Applying Lemma~\ref{lem-enhance} completes the proof of Proposition~\ref{prop-crossing-dimension} (since $(\alpha(1 - (\frac{\alpha - 1}{10} \wedge \frac{1}{10}))>1$).
\end{proof}

\section{Proof of Theorem~\ref{thm-main}}\label{sec-main-thm}

Let $\alpha>1$ be as in Proposition~\ref{prop-crossing-dimension}. Let $\sqrt{1/\alpha} < \alpha' < 1$. 
\begin{lemma}\label{lem-m-star}
For $N^\star \geq 16$, set $K = (N^\star)^{\alpha \alpha'}$ and $\Delta = (N^\star)^{-\alpha(\alpha')^2}$, and let $\tilde h^{(N)}$ be defined as in \eqref{eq-def-tilde-h}. Write $m^\star_N = m^\star_N(N^\star) =  \P(o \in \mathcal C_*^{\Lambda_N})$. Then there exists $C= C(\epsilon)>0$ such that  $m^\star_{N^\star} \leq C (N^\star)^{-6}$.
\end{lemma}
\begin{remark}
In this lemma, regardless of the size of the box under consideration, the amount of perturbation $\Delta$ in our field $\tilde h^{(N)}$ only depends on $N^\star$. This is crucial for  \eqref{eq-monotonicity-consequence} below.
\end{remark}
\begin{proof}
It suffices to show that by recursion, there exists $N_0 = N_0(\epsilon)$ such that  for  $N^\star \geq N_0$
\begin{equation}\label{eq-m-recursion-star}
m^\star_{2N} \leq K^{-\frac{1-\alpha'}{2}} m^\star_{N/2} \mbox{ for }  (N^\star)^{\alpha'} \leq  N \leq N^\star\,.
\end{equation}
Suppose that \eqref{eq-m-recursion-star} fails for some  $(N^\star)^{\alpha'} \leq  N \leq N^\star$. Since $\Lambda_N \subset v + \Lambda_{2N}$ for all $v\in \Lambda_{N/4}$ and $v + \Lambda_{N/2} \subset \Lambda_N$ for all $v\in \mathcal A_{N/2}$,  by \eqref{eq-monotonicity} we see 
\begin{equation}\label{eq-monotonicity-consequence}
\E|\mathcal C_*^{\Lambda_N} \cap \Lambda_{N/4}|  \geq \frac{N^2}{32} m^\star_{2N} \mbox{ and } \E |\mathcal C_*^{\Lambda_N} \cap \mathcal A_{N/2}| \leq N^2 m^\star_{N/2}\,.
\end{equation}
Together with the assumption that \eqref{eq-m-recursion-star} fails, this yields that 
$$ \E|\mathcal C_*^{\Lambda_N} \cap \Lambda_{N/4}|  > 32^{-1} K^{-\frac{1-\alpha'}{2}}\E |\mathcal C_*^{\Lambda_N} \cap \mathcal A_{N/2}|\,.$$
Since $|\mathcal C_*^{\Lambda_N} \cap \Lambda_{N/4}| $ and $|\mathcal C_*^{\Lambda_N} \cap \mathcal A_{N/2}|$ are integer-valued and are at most $N^2$, the preceding inequality implies that
$$ \P(|\mathcal C_*^{\Lambda_N} \cap \Lambda_{N/4}|  > 32^{-1} K^{-\frac{1-\alpha'}{2}} |\mathcal C_*^{\Lambda_N} \cap \mathcal A_{N/2}|) \geq \frac{1}{32 N^3}\,.$$
Now, set $N_0 = N_0(\epsilon)$ sufficiently large so that 
\begin{equation}\label{eq-N-0}
\frac{1}{10^6 N^3} > \kappa^{-1} e^{- N^{\kappa}} \mbox{and } 32^{-1} K^{-\frac{1-\alpha'}{2}} > \frac{8}{K\Delta} \mbox{ for all } N\geq (N_0)^{\alpha'}\,.
\end{equation} 
Therefore, by Proposition~\ref{prop-crossing-dimension}, there exists at least one instance such that 
$$|\mathcal C_*^{\Lambda_N} \cap \Lambda_{N/4}|  > 32^{-1} K^{-\frac{1-\alpha'}{2}} |\mathcal C_*^{\Lambda_N} \cap \mathcal A_{N/2}|\mbox{ and } d_{\mathcal C_*^{\Lambda_N}}(\partial \Lambda_{N/4}, \partial \Lambda_{N/2}) \geq K\,.$$
This contradicts Lemma~\ref{lem-perturbation}, thus completing the proof of the lemma.
\end{proof}

In the proof of Lemma~\ref{lem-m-N-bound} below, it is important for us to have independence between different scales. To this end, it is useful to consider a perturbation which only occurs in an annulus.
\begin{lemma}\label{lem-m-annulus}
Let $\Delta = (N/16)^{-\alpha(\alpha')^2}$ and define 
\begin{equation}\label{eq-def-tilde-h-annulus}
\tilde h^{(N)}_v = 
\begin{cases}
h_v + \Delta &\mbox{ for } v\in \mathcal A_{N/4}\,, \\
 h_v &\mbox{ for }v\not\in \mathcal A_{N/4}\,.
\end{cases}
\end{equation}
Then there exists $C=C(\epsilon)>0$ such that $\P(o\in \mathcal C_*^{\Lambda^N}) \leq C N^{-5}$. 
\end{lemma}
\begin{proof}
For $v\in \partial \Lambda_{3N/16}$, let $B_v$ be a translated copy of $\Lambda_{N/16}$ centered at $v$. 
Thus, for all $u \in B_v$ we have $\tilde h^{(N)}_u = h_u + (N/16)^{-\alpha(\alpha')^2}$. Recall $m^\star_{N/16}(N/16)$ as in Lemma~\ref{lem-m-star}.  By \eqref{eq-monotonicity} and Lemma~\ref{lem-m-star}, 
$$\P(v\in \mathcal C_*^{\Lambda_N}) \leq m^\star_{N/16}(N/16) \leq C N^{-6}\,.$$
 Hence, 
$\P(\partial \Lambda_{3N/16} \cap \mathcal C_*^{\Lambda_N} \neq \emptyset) \leq C N^{-5}$ by a simple union bound. Combined with Lemma~\ref{lem-percolation-property} (and the simple observation that $o$ cannot percolate in $\mathcal C_*^{\Lambda_N}$ to $\partial \Lambda_N$ if $\partial \Lambda_{3N/16} \cap \mathcal C_*^{\Lambda_N} = \emptyset$), this completes the proof of the lemma.
\end{proof}

\begin{lemma}\label{lem-m-N-bound}
There exists $C= C(\epsilon)>0$ such that  $m_N \leq C N^{-3}$.
\end{lemma}
\begin{proof}
Without loss of generality let us only consider  $N = 8^n$ for some $n\geq 1$, and define $\tilde h^{(N)}_v$ as in \eqref{eq-def-tilde-h-annulus}.
 Let $E_\ell = \{o\not\in \mathcal C_*^{\Lambda_{8^{\ell}}}\}$ and $E = \cap_{0.9n \leq \ell \leq n} E_\ell$. By Lemma~\ref{lem-m-annulus}, we see that $\P(E^c) \leq CN^{-3}$ for some $C = C(\epsilon)>0$ (whose value may be adjusted later in the proof). For $0.9 n\leq \ell \leq n$, let $\mathcal F_\ell = \sigma(h_v: v\in \Lambda_{8^\ell})$ and write 
\begin{equation}\label{eq-Gaussian-conditioning}
h_v = (|\mathcal A_{2\cdot 8^{\ell}}|)^{-1} h_{\mathcal A_{2\cdot 8^{\ell}}} + g_v \mbox{ for } v\in \mathcal A_{2\cdot 8^{\ell}}\,,
\end{equation}
where $\{g_v: v \in \mathcal A_{2\cdot 8^{\ell}}\}$ is a mean-zero Gaussian process independent of $h_{\mathcal A_{2\cdot 8^{\ell}}}$ and $\{g_v: v \in \mathcal A_{2\cdot 8^{\ell}}\}$ for $0.9 n\leq \ell \leq n$ are mutually independent. Let $\mathcal F'_\ell$ be the $\sigma$-field which contains every event in $\mathcal F_{\ell+1}$ that is independent of $h_{\mathcal A_{2 \cdot 8^{\ell}}}$ (so in particular $\mathcal F_\ell \subset \mathcal F'_\ell$). By monotonicity,  there exists an interval $I_\ell$ measurable with respect to $\mathcal F'_\ell$ such that conditioned on $\mathcal F'_\ell$ we have $o \in \mathcal C^{\Lambda_{8^{\ell+1}}}$ if and only if $h_{\mathcal A_{2 \cdot 8^{\ell}}} \in I_\ell$. Let $I'_\ell$ be the maximal sub-interval of $I_\ell$ which shares the upper endpoint and $|I'_\ell| \leq  \frac{|\mathcal A_{2\cdot 8^\ell}| \cdot 16}{8^{\ell \alpha (\alpha')^2}}$.
By our definition of $E_{\ell+1}$, we see
 from \eqref{eq-Gaussian-conditioning} that  conditioned on $\mathcal F'_\ell$ we have $\{o \in \mathcal C^{\Lambda_{8^{\ell+1}}}\} \cap E_{\ell+1}$ if and only if $h_{\mathcal A_{2 \cdot 8^{\ell}}} \in I'_\ell$. Thus, for $0.9 n \leq \ell \leq n$, 
$$\P(\{o \in \mathcal C^{\Lambda_{8^{\ell+1}}}\} \cap E_{\ell+1} \mid \mathcal F'_\ell) \leq \P(h_{\mathcal A_{2\cdot 8^\ell}}\in I'_\ell)\,.$$
Combined with the fact that $\var (h_{\mathcal A_{2\cdot 8^\ell}}) = \epsilon^2 |\mathcal A_{2\cdot 8^\ell}|$, this gives that
$$\P(\{o \in \mathcal C^{\Lambda_{8^{\ell+1}}}\} \cap E_{\ell+1} \mid \mathcal F'_\ell) \leq \frac{C}{8^{\ell(\alpha (\alpha')^2-1)}}\,.$$
Since $\{o \in \mathcal C^{\Lambda_{8^{n}}}\} \cap E = \cap_{\ell = 0.9n}^{n-1} (\{o \in \mathcal C^{\Lambda_{8^{\ell+1}}}\} \cap E_{\ell+1})$ and since $\{o \in \mathcal C^{\Lambda_{8^{\ell}}}\} \cap E_{\ell}$ is $\mathcal F_\ell$-measurable (and thus is $\mathcal F'_\ell$-measurable), we deduce that $\P(o\in \mathcal C^{\Lambda_N} \cap E) \leq C N^{-3}$. Combined with the fact that $\P(E^c) \leq CN^{-3}$, it completes the proof of the lemma.
\end{proof}

\begin{proof}[Proof of Theorem~\ref{thm-main}]
Let $N_0 = N_0(\epsilon)$ be chosen later. For $B\in \mathcal B(N, N_0)$, we say $B$ is open if $\mathcal C^{B^{\mathrm{large}}} \cap B \neq \emptyset$. Clearly, this percolation process satisfies the $(N, N_0, 4, p)$-condition where 
\begin{equation}\label{eq-p-bound}
p  = \P(\mathcal C^{B^{\mathrm{large}}} \cap B \neq \emptyset) \leq N_0^2 m_{N_0/2} \leq CN_0^{-1}  \mbox{ for } C = C(\epsilon)>0\,.
\end{equation}
(The last transition above follows from Lemma~\ref{lem-m-N-bound}.)
In addition, we note that in order for $o\in \mathcal C^{\Lambda_N}$, it is necessary that there exists an open lattice animal on $B\in \mathcal B(N, N_0)$ with size at least $\frac{N}{10N_0}$. Now, choosing $N_0$ sufficiently large (so that $p$ is sufficiently small, by \eqref{eq-p-bound}) and applying Lemma~\ref{lem-enhance} completes the proof.
\end{proof}

\noindent {\bf Acknowledgement.} We thank Tom Spencer for introducing the problem to us, thank Steve Lalley for many interesting discussions and thank Subhajit Goswami, Steve Lalley for a careful reading of an earlier version of the manuscript.

\end{document}